\documentclass[12pt,showkeys]{amsart}
\usepackage{color}
\usepackage{graphicx}
\usepackage{subfigure}
\usepackage{amssymb}
\usepackage{amsmath}
\usepackage{multirow}
\usepackage{enumerate}
\usepackage{epstopdf}
\usepackage{cite,stmaryrd,txfonts}
\usepackage{tikz}

\input{undertilde}
\usetikzlibrary{snakes}

\usepackage{verbatim}

\usepackage{epic}

\usepackage{empheq}

\usepackage{ulem}

\newcommand{\bs}{\boldsymbol}

\def\cequiv{\raisebox{-1.5mm}{$\;\stackrel{\raisebox{-3.9mm}{=}}{{\sim}}\;$}}

\def\uphi{\undertilde{\varphi}}
\def\upsi{\undertilde{\psi}}

\def\uf{\undertilde{f}}

\def\ur{\undertilde{r}}
\def\us{\undertilde{s}}

\def\uu{\undertilde{u}}

\def\uv{\undertilde{v}}
\def\ux{\undertilde{x}}

\def\curl{{\rm curl}}
\def\dv{{\rm div}}

\def\vgm12{\bs{V}^{1+,2}_{\gamma,M}}

\newtheorem{theorem}{Theorem}
\newtheorem{remark}[theorem]{Remark}

\newtheorem{lemma}[theorem]{Lemma}

\newcounter{mnote}
\let\oldmarginpar\marginpar
\renewcommand\marginpar[1]{\-\oldmarginpar[\raggedleft\footnotesize #1]%
  {\raggedright\footnotesize #1}}

\begin{document}

\title{Amiable mixed schemes for fourth order curl equations}
\author{Shuo Zhang}
\address{LSEC, Institute of Computational Mathematics and Scientific/Engineering Computing, Academy of Mathematics and System Sciences, Chinese Academy of Sciences, Beijing 100190, People's Republic of China}
\email{szhang@lsec.cc.ac.cn}
\thanks{The author is supported partially by the National Natural Science Foundation of China with Grant No. 11471026 and National Centre for Mathematics and Interdisciplinary Sciences, Chinese Academy of Sciences.}




\begin{abstract} 
In this paper, amiable mixed schemes are presented for two variants of fourth order curl equations. Specifically, mixed formulations for the problems are constructed, which are well-posed in Babu\v{s}ka-Brezzi's sense and admit stable discretizations by finite element spaces of low smoothness and of low degree. The regularities of the mixed formulations and thus equivalently the primal problems are established, and some finite elements examples are given which can exploit the regularity of the solutions to an optimal extent.
\end{abstract}

\maketitle


%
%
%
\section{Introduction}

In this paper, we study the boundary value problem of the fourth order curl operator of type
\begin{equation}\label{eq:bvpA}
{\bf(A)}\quad\left\{
\begin{array}{rl}
(\nabla\times)^4\uu=\uf,&\mbox{in}\ \Omega;
\\
\nabla\cdot\uu=0 & \mbox{in}\ \Omega;
\\
\uu\times\mathbf{n}=\undertilde{0},\ (\nabla\times \uu)\times \mathbf{n}=\undertilde{0} & \mbox{on}\ \partial\Omega,
\end{array}
\right.
\end{equation}
where $\dv\uf=0$, and of variant type
\begin{equation}\label{eq:bvpB}
{\bf(B)}\quad\left\{
\begin{array}{rl}
(\nabla\times)^4\uu+\uu=\uf,&\mbox{in}\ \Omega;
\\
\uu\times\mathbf{n}=\undertilde{0},\ (\nabla\times \uu)\times \mathbf{n}=\undertilde{0} & \mbox{on}\ \partial\Omega.
\end{array}
\right.
\end{equation}
For \eqref{eq:bvpB}, it is not necessary that $\dv\uf=0$. But evidently, $\dv\uu=0$ when $\dv\uf=0$. 

The boundary value problem of fourth order curl operator $(\nabla\times)^4$ arises in different applications, like in magnetohydrodynamics(MHD) and in the inverse electromagnetic scattering theory. In MHD, $(\nabla\times)^4\bf{B}$ is involved in the resistive system where $\bf{B}$ is the magnetic field as a primary variable\cite{Zheng.B;Hu.Q;Xu.J2011}, and in the inverse electromagnetic scattering theory, $(\nabla\times)^4$ appears in computing the transmission eigenvalue\cite{Cakoni.F;Haddar.H2007}. Some more applications of $(\nabla\times)^4$ can be found in the sequel works. The divergence free condition and the boundary conditions as in \eqref{eq:bvpA} are generally used. The two types of equations as above were discussed in, e.g., \cite{Sun.J2016,Nicaise.S2016} and \cite{Zheng.B;Hu.Q;Xu.J2011,Hong.Q;Hu.J;Shu.S;Xu.J2012}, respectively.

There have been works devoted to the discretization of the model problems. For the primal variational formulation, two kinds of not-conforming discretizations are discussed in literature, including a nonconforming element constructed in Zheng-Hu-Xu \cite{Zheng.B;Hu.Q;Xu.J2011}, and discretizations  given in Hong-Hu-Shu-Xu \cite{Hong.Q;Hu.J;Shu.S;Xu.J2012} with standard high order Nedelec elements in the framework of discontinuous Galerkin method. So far, no curl-curl-conforming other than the $H^2$-conforming finite element is known to us. An alternative approach is to introduce and deal with mixed/order-reduced formulation. It is natural to consider possibly the operator splitting technique which introduces an intermediate variable and then reduce the original problem to a system of second order equations. This is the way adopted by Sun \cite{Sun.J2016}. The associated eigenvalue problem is also discussed therein. Beyond these discussions, few results on the discretisation are known to us.

The operator $(\nabla\times)^4$ is of fourth order and not completely symmetric; this makes the model problems bear complicated intrinsic structure. Primarily, the multiple high stiffness effects the property of the problems. Very recently, Nicaise \cite{Nicaise.S2016} studies the boundary value problem \eqref{eq:bvpA}, and proves that the solution does not generally belong to $\undertilde{H}^3(\Omega)$ on polyhedons, and the $\undertilde{H}^2(\Omega)$ regularity of the solution of \eqref{eq:bvpA} is still open even on convex polyhedrons. The multiple high stiffness also makes the concentrative construction of finite element functions difficult. Moreover, the structure of the finite element spaces are very complicated, which makes designing optimal solvers/multilevel methods difficult; there has been no discussion along this line. The order-reduced discretisation scheme by \cite{Sun.J2016} enables to utilise the existing edge element to solve the original problem; this scheme can be viewed as an analogue of the Ciarlet-Raviart's scheme \cite{Ciarlet.P;Raviart.P1974} for biharmonic equation in the context of fourth order curl problem. However, the structure has not become friendlier with this formulation. The stability analysis has not been presented in \cite{Sun.J2016}, and thus the intrinsic topology is not clear and the convergence analysis is constructed in quite a technical way there. Further, the internal structure of the finite element space is not yet clear either, and thus designing optimal solvers/multilevel methods for the scheme is also difficult.

At the current situation, in this paper, we introduce new mixed formulations to figure out and utilise a clear and amiable structure. By bringing in auxiliary variables for the problems {\bf(A)} and {\bf(B)}, we present mixed formulations which are stable in Babu\v{s}ka-Brezzi's sense on the spaces of $L^2$, $H(\curl)$ and $H^1$ types. Also, we establish the regularity results for the mixed formulations on convex polyhedral domains. As the mixed formulations are equivalent to the primal ones, the $\undertilde{H}{}^2(\Omega)$ regularity of $\uu$ and $\nabla\times\uu$ are confirmed for \eqref{eq:bvpA} and \eqref{eq:bvpB} on convex polyhedrons, and the assumptions adopted in \cite{Sun.J2016} and \cite{Zheng.B;Hu.Q;Xu.J2011} are confirmed.  The mixed formulations admit amiable discretisation with finite element spaces corresponding to $L^2$, $H(\curl)$ and $H^1$ under some mild conditions, and the theoretical convergence analysis can be done in a standard friendly way. Several finite element examples are presented which can exploit the regularity of the solutions to an optimal extent. As the structures of the dicsretized $L^2$, $H(\curl)$ and $H^1$ spaces have been well-studied, the newly-developed discretisation scheme can be solved optimally by the aid of some existing optimal preconditioners\cite{Hiptmair.R;Xu.J2007,Xu.J1992,Xu.J2010,Ruesten.T;Winther.R1992}. Moreover, it is easy to find finite element spaces that are nested on nested grids, both algebraically and topologically with respect to the mixed formulation; this can bring convenience in designing further high-efficiency algorithms. 

We would emphasize the new variational problems \eqref{eq:nvpA}(for \eqref{eq:bvpA}) and \eqref{eq:nvpB}(for \eqref{eq:bvpB}) are the starting point of what we are going to do and what we are able to do. These new primal formulations arise from configurating the essential boundary conditions that should be satisfied by the solutions, and they differ from traditional ones, like ones discussed in \cite{Hong.Q;Hu.J;Shu.S;Xu.J2012}, \cite{Nicaise.S2016} or \cite{Sun.J2016}. The variational formulation \eqref{eq:nvpB} is similar to the one used in \cite{Zheng.B;Hu.Q;Xu.J2011}, but the original boundary condition discussed in \cite{Zheng.B;Hu.Q;Xu.J2011} is different from that of \eqref{eq:bvpB}. The new variational formulations possess enough capacity for the essential boundary conditions, and make the sequel analysis smoother. 

The remaining of the paper is organised as follows. In Section \ref{sec:pre}, we present some preliminaries and the model problems in the primal formulation. We will particularly figure out the appropriate spaces of the model problem by clarifying the boundary conditions and specify the space whose capacity is big enough for the boundary condition and the variational form. In Section \ref{sec:mf}, the mixed formulation of the model problems are given with stability analysis. Section \ref{sec:dis} is then devoted to the discretizations, including general discussion on the  conditions to be satisfied, and also some specific examples. Finally in Section \ref{sec:con}, concluding remarks are given.

\section{Model problems: New primal formulations}
\label{sec:pre}

\subsection{Preliminaries: Sobolev spaces and finite elements}

Let $\Omega\subset\mathbb{R}^3$ be a simply connected polyhedral domain, with boundary $\Gamma=\partial\Omega$, and unit outward norm vector $\mathbf{n}$. In this paper, we use the bold symbol for a vector in $\mathbb{R}^3$, and a subscript $\undertilde{~}$ for a vector valued function. We assume $\Gamma$ is also connected. We use $L^2(\Omega)$ and $H^t_{(0)}(\Omega)$ for $t=1,2,\dots$ for the standard Lebesque space and Sobolev spaces. Denote
\begin{eqnarray}
H^s(\curl,\Omega):=\{\uv\in (L^2(\Omega))^3:\curl^j\uv\in (L^2(\Omega))^3,\ 1\leqslant j\leqslant t\},\ t=1,2,\dots, 
\end{eqnarray}
equipped with the inner product $(\uu,\uv)_{H^t(\curl,\Omega)}=(\uu,\uv)+\sum_{j=1}^t(\curl^j\uu,\curl^j\uv),$ and the corresponding norm $\|\cdot\|_{H^t(\curl,\Omega)}$. Particularly, $H^1(\curl,\Omega)=H(\curl,\Omega)$, and $\|\cdot\|_{H(\curl,\Omega)}=\|\cdot\|_{\curl,\Omega}$. Similarly, define
\begin{equation}
H(\curl^2,\Omega):=\{\uv\in (L^2(\Omega))^3:\curl\curl\uv\in(L^3(\Omega))^2\},
\end{equation}
equipped with the inner product $(\uu,\uv)_{H(\curl^2,\Omega)}=(\uu,\uv)+(\curl\curl\uu,\curl\curl\uv)$ and the corresponding norm. Corresponding to the boundary condition, define $H^2_0(\curl,\Omega):=\{\uv\in H^2(\curl,\Omega):\uv\times\mathbf{n}=\undertilde{0}\ \mbox{and}\ (\curl\uv)\times\mathbf{n}=\undertilde{0}\ \mbox{on}\ \Gamma\}$. Define $H(\dv,\Omega)=\{\uv\in (L^2(\Omega))^3:\dv\uv\in L^2(\Omega)\}$, and $H_0(\dv,\Omega)=\{\uv\in H(\dv,\Omega):\uv\cdot\mathbf{n}=0\,\mbox{on}\,\Gamma\}$. In the sequel, we use $\nabla\times$ for $\curl$ in equations. 

\begin{lemma}\label{lem:picurl}
There exists a constant $C$, such that it holds for $\uv\in H_0(\curl,\Omega)$ and $\dv\uv=0$ that
\begin{equation}
\|\uv\|_{0,\Omega}\leqslant C\|\uv\|_{\curl,\Omega}.
\end{equation}
\end{lemma}
\begin{lemma}\label{lem:hongetal}[Lemma 2.1 of \cite{Hong.Q;Hu.J;Shu.S;Xu.J2012}]
$H^2_0(\curl,\Omega)$ is the closure of $(\mathcal{C}^\infty_0(\Omega))^3$ in $H(\curl^2,\Omega)$, and 
$$
\|\nabla\times\uv\|_{0,\Omega}\leqslant \frac{1}{2}(\|\nabla\times\nabla\times\uv\|_{0,\Omega}+\|\uv\|_{0,\Omega})\ \ \mbox{on}\ \,H^2_0(\curl,\Omega).
$$
\end{lemma}

Let $\Omega$ be subdivided to tetrahedrons which form a grid $\mathcal{G}_h$. We impose the shape regularity assumption on $\mathcal{G}_h$. On the grid can finite element spaces be constructed. We refer to \cite{GiraultRaviart1986,Monk.P2003,Ciarlet.P1978} for the context of finite element methods. We only recall these familiar finite element spaces:
\begin{itemize}
\item continuous Lagrangian element space of $k$-th degree: subspace of $H^1(\Omega)$, consist of piecewise $k$-th-degree polynomials; denoted by $\mathcal{L}^k_{h(0)}$, without or with respect to the $H^1_0(\Omega)$ boundary condition;
\item Nedelec edge element of first family of $k$-th degree: subspace of $H(\curl,\Omega)$, consist of piecewise polynomials of the form $\uu+\uv$, with $\uu\in (P_{k-1})^3$ and $\uv\in \ux\times(\hat{P}_{k-1})^3$, where $P_{k-1}$ is the space of $(k-1)$-th degree polynomials, and $\hat{P}_{k-1}^3$ is the space of homogeneous $(k-1)$-th degree polynomials; denoted by $N^k_{h(0)}$, without or with respect to the $H_0(\curl,\Omega)$ boundary condition. 
\end{itemize}
Particularly, $\mathcal{L}^0_{h}$ denotes the space of piecewise constant, and $\mathcal{L}^0_{h0}=\mathcal{L}^0_{h}\cap L^2_0(\Omega)$. The fact below is well known.
\begin{lemma} 
$\nabla \mathcal{L}^k_{h(0)}=\{\uv\in N^k_{h(0)}:\nabla\times\uv=\undertilde{0}\}$.
\end{lemma}

\subsection{New primal formulation} 
The variational formulations of the boundary value problem are used in literature:
\begin{itemize}
\item for {\bf (A)} (\cite{Sun.J2016}, e.g.,): given $\uf$ with $\dv\uf=0$, find $\uu\in H^2_0(\curl,\Omega)$ and $\dv\uu=0$, such that 
\begin{equation}\label{eq:vfA}
(\nabla\times\nabla\times \uu,\nabla\times\nabla\times\uv)=(\uf,\uv),\ \forall\,\uv\in H^2_0(\curl,\Omega);
\end{equation}
\item for {\bf (B)} (\cite{Hong.Q;Hu.J;Shu.S;Xu.J2012}, e.g.,): given $\uf$, to find $\uu\in H^2_0(\curl,\Omega)$, such that 
\begin{equation}\label{eq:vfB}
(\nabla\times\nabla\times \uu,\nabla\times\nabla\times\uv)+(\uu,\uv)=(\uf,\uv),\ \forall\,\uv\in H^2_0(\curl,\Omega).
\end{equation}
\end{itemize}
The well-posedness of the two variational problems is guaranteed by Lemmas \ref{lem:picurl} and \ref{lem:hongetal}.

Note that $\nabla\times\uu\in H_0(\dv,\Omega)$ for $\uu\in H_0(\curl,\Omega)$, and thus the boundary control of $\nabla\times\uu$ is more than the capacity of boundary condition of $H_0(\curl,\Omega)$. This way, we set up the variational form on another Sobolev space. We begin with the fact below.
\begin{lemma}\label{lem:h=cd}[Lemma 2.5, \cite{GiraultRaviart1986}]
If $\Omega\subset\mathbb{R}^3$ is bounded, simply connected with Liptschiz-continuous boundary, then 
$$
\undertilde{H}{}^1_0(\Omega):=(H^1_0(\Omega))^3=H_0(\curl,\Omega)\cap H_0(\dv,\Omega).
$$
and
$$
(\nabla\uu,\nabla\uv)=(\nabla\times\uu,\nabla\times\uv)+(\dv\uu,\dv\uv)\ \ \mbox{for}\ \uu,\uv\in \undertilde{H}{}^1_0(\Omega).
$$
\end{lemma}

Now, define $\undertilde{H}{}^1_0(\curl,\Omega):=\{\uv\in H_0(\curl,\Omega):\nabla\times\uv\in \undertilde{H}{}^1_0(\Omega)\}$. The observation below is crucial. 
\begin{lemma}
$\undertilde{H}{}^1_0(\curl,\Omega)=H^2_0(\curl,\Omega)$.
\end{lemma}
\begin{proof}
Evidently, $\undertilde{H}{}^1_0(\curl,\Omega)\subset H^2_0(\curl,\Omega)$. On the other hand, given $\uv\in H^2_0(\curl,\Omega)$, $\nabla\times\uv\in H_0(\curl,\Omega)\cap H_0(\dv,\Omega)=\undertilde{H}{}^1_0(\Omega)$. This finishes the proof.
\end{proof}
Moreover, $(\nabla\nabla\times\uu,\nabla\nabla\times\uv)=(\nabla\times\nabla\times\uu,\nabla\times\nabla\times\uv)$ on $\undertilde{H}{}^1_0(\curl,\Omega)$. Therefore, we establish the  variational form of the primal model problems  as:
\begin{enumerate}[{\bf(A$'$)}]
\item Given $\uf$ with $\dv\uf=0$, find $\undertilde{u}\in \undertilde{H}{}^1_0(\curl,\Omega)$, $\dv\,\uu=0$, such that 
\begin{equation}\label{eq:nvpA}
(\nabla\nabla\times\uu,\nabla\nabla\times \uv)=(\uf,\uv),\ \ \forall\,\uv\in \undertilde{H}{}^1_0(\curl,\Omega).
\end{equation}

\item Given $\uf$, find $\undertilde{u}\in \undertilde{H}{}^1_0(\curl,\Omega)$, such that 
\begin{equation}\label{eq:nvpB}
(\nabla\nabla\times\uu,\nabla\nabla\times \uv)+(\uu,\uv)=(\uf,\uv),\ \ \forall\,\uv\in \undertilde{H}{}^1_0(\curl,\Omega).
\end{equation}
\end{enumerate}

\begin{lemma}
The variational problems {\bf(A$'$)} and {\bf(B$'$)} are well-posed. They are equivalent to \eqref{eq:vfA} and \eqref{eq:vfB}, respectively.
\end{lemma}

\begin{remark}
The variational problem on $\undertilde{H}{}^1_0(\curl,\Omega)$  has indeed been used in \cite{Zheng.B;Hu.Q;Xu.J2011}, where the original boundary condition is $\uu\times\mathbf{n}=\nabla\times\uu=\undertilde{0}$. In this paper here, we show that even when the boundary condition is simplified, the space $\undertilde{H}{}^1_0(\curl,\Omega)$ is still the appropriate one. 
\end{remark}

\section{Mixed formulation of model problems}
\label{sec:mf}

In this section, we present mixed problems that are equivalent to {\bf (A$'$)} and {\bf (B$'$)}, thus to {\bf (A)} and {\bf (B)}, respectively. The stability and regularity of the mixed problems are given. 

\subsection{Mixedization of Problem {\bf(A $'$)}}

We start with the observations below. Let $\uu$ be the solution of {\bf (A$'$)} (thus {\bf (A)}). Define $\uphi:=\nabla\times\uu$, then $\uphi\in\undertilde{H}{}^1_0(\Omega)$, and $\dv\uphi=0$. On the other hand, given $\uphi\in\undertilde{H}{}^1_0(\Omega)$, $\uphi=\curl\uu$ for $\uu\in H_0(\curl,\Omega)$ iff $(\dv\uphi,q)=0$ for any $q\in L^2_0(\Omega)$ and $(\uphi,\nabla\times\us)=(\nabla\times\uu,\nabla\times\us)$ for any $\us\in H_0(\curl,\Omega)$. Also, $\uu$ is uniquely determined by the divergence free condition, which reads equivalently $(\uu,\nabla h)=0$ for any $h\in H_0(\Omega)$. Now define 
\begin{equation}
V:=H^1_0(\Omega)\times H_0(\curl,\Omega)\times\undertilde{H}{}^1_0(\Omega)\times L^2_0(\Omega)\times H_0(\curl,\Omega)\times H^1_0(\Omega).
\end{equation} 
The mixed formulation of  {\bf (A)} is to find $(m,\uu,\uphi,p,\ur,g)\in V$, such that for $\forall\,(n,\uv,\upsi,q,\us,h)\in V$, 
\begin{equation}\label{eq:mxdvfA'}
\left\{
\begin{array}{ccccccll}
&&&&(\ur,\nabla n)&& = 0, 
\\
&&&&-(\nabla\times\ur,\nabla\times\uv)&+(\nabla g,\uv)& =(\uf,\uv),
\\
&&(\nabla\uphi,\nabla\upsi)&+(p,\dv\upsi) &+(\nabla\times\ur,\upsi)&  & =0, 
\\
&&(\dv\uphi,q)&&&&=0, 
\\
(\nabla m,\us)&-(\nabla\times\uu,\nabla\times\us)& (\uphi,\nabla\times\us) & &&&=0, 
\\
&(\uu,\nabla h)&&&&&=0.
\end{array}
\right.
\end{equation}

\begin{lemma}\label{lem:wellposeA'}
Given $\uf\in \undertilde{L}{}^2(\Omega)$, the problem \eqref{eq:mxdvfA'} admits a unique solution $(m,\uu,\uphi,p,\ur,g)\in V$. Moreover, 
\begin{equation}
\|m\|_{1,\Omega}+\|\uu\|_{\curl,\Omega}+\|\uphi\|_{1,\Omega}+\|p\|_{0,\Omega}+\|\ur\|_{\curl,\Omega}+\|g\|_{1,\Omega}\leqslant C\|\uf\|_{(H_0(\curl,\Omega))'}.
\end{equation}
\end{lemma}
\begin{proof}
We are going to verify conditions by Brezzi's theory. Define
\begin{equation}\label{eq:aforA}
a((m,\uu,\uphi),(n,\uv,\upsi)):=(\nabla\uphi,\nabla\upsi),
\end{equation}
and
\begin{equation}\label{eq:bforA}
b((m,\uu,\uphi),(q,\us,h)):=(\dv\uphi,q)+(\nabla m,\us)-(\nabla\times\uu,\nabla\times\us)+(\uphi,\nabla\times\us)+(\uu,\nabla h).
\end{equation}
Then $a(\cdot,\cdot)$ and $b(\cdot,\cdot)$ are continuous on $[H^1_0(\Omega)\times H_0(\curl,\Omega)\times \undertilde{H{}^1_0(\Omega)}]^2$ and $[H^1_0(\Omega)\times H_0(\curl,\Omega)\times \undertilde{H}{}^1_0(\Omega)]\times[L^2_0(\Omega)\times H_0(\curl,\Omega)\times H^1_0(\Omega)]$, respectively. Define $Z:=\{(m,\uu,\uphi)\in [H^1_0(\Omega)\times H_0(\curl,\Omega)\times \undertilde{H}{}^1_0(\Omega)]: b((m,\uu,\uphi),(q,\us,h))=0,\ \forall\,(q,\us,h)\in [L^2_0(\Omega)\times H_0(\curl,\Omega)\times H^1_0(\Omega)]\}$. Then it remains for us to verify the coercivity of $a(\cdot,\cdot)$ on $Z$ and inf-sup condition: given nonzero $(q,\us,h)\in L^2_0(\Omega)\times H_0(\curl,\Omega)\times H^1_0(\Omega)$,
\begin{equation}
\sup_{(m,\uu,\uphi)\in H^1_0(\Omega)\times H_0(\curl,\Omega)\times \undertilde{H}{}^1_0(\Omega)}\frac{b((m,\uu,\uphi),(q,\us,h))}{\|m\|_{1,\Omega},\|\uu\|_{\curl,\Omega},\|\uphi\|_{1,\Omega}}\geqslant C(\|q\|_{0,\Omega}+\|\us\|_{\curl,\Omega}+\|h\|_{1,\Omega}).
\end{equation}
Given $(m,\uu,\uphi)\in Z$, then $m=0$. Since $(\uu,\nabla h)=0$ for any $h\in H^1_0(\Omega)$, we have $\|\uu\|_{\curl,\Omega}\leqslant C\|\nabla\times\uu\|_{0,\Omega}$. Since $(\nabla\times\uu,\nabla\times\uu)=(\uphi,\nabla\times\uu)$, we have $\|\nabla\times\uu\|_{0,\Omega}\leqslant \|\uphi\|_{0,\Omega}\leqslant C\|\nabla\uphi\|_{0,\Omega}$. This confirms the coercivity of $a(\cdot,\cdot)$ on $Z$.

Given $(q,\us,h)\in L^2_0(\Omega)\times H_0(\curl,\Omega)\times H^1_0(\Omega)$, firstly, we decomose $\us=\us{}_1+\us{}_2$, such that $\us{}_1\in\nabla H^1_0(\Omega)$, and $\us{}_2\in (\nabla H^1_0(\Omega))^\perp$. Set $\uphi$ to be such that $(\dv\uphi,q)=(q,q)$ and $\|\uphi\|_{1,\Omega}\leqslant C\|\dv\uphi\|_{0,\Omega}$, $m$ to be such that $\nabla m=\us{}_1$. Further, $\uu$ is chosen to be $\uu{}_1+\nabla h$, such that $(\uu{}_1,\nabla g)=0$ for any $g\in H^1_0(\Omega)$ and $(\uphi-\nabla\uu{}_1,\nabla\times\uv)=(\nabla\times\us,\nabla\times\uv)$ for any $\uv\in H_0(\curl,\Omega)$. Then 
$$
b((m,\uu,\uphi),(q,\us,h))=(q,q)+(\us{}_1,\us{}_1)+(\curl\us{}_2,\curl\us{}_2)+(\nabla h,\nabla h)\geqslant C(\|q\|_0^2+\|\us\|_{\curl,\Omega}^2+\|\nabla h\|_{0,\Omega}^2).
$$
Meanwhile, $\|\nabla m\|_{0,\Omega}=\|\us{}_1\|_{0,\Omega}\leqslant \|\us\|_{\curl,\Omega}$, $\|\uphi\|_{1,\Omega}\leqslant C\|q\|_{0,\Omega}$, and $\|\uu\|_{\curl,\Omega}\leqslant C(\|\nabla h\|_{0,\Omega}+\|\us\|_{\curl,\Omega})$. This constructs the inf-sup condition and finishes the proof.
\end{proof}

\begin{lemma}\label{lem:equiv}
The problem \eqref{eq:mxdvfA'} is equivalent to the variational problem {\bf (A$'$)}.
\end{lemma}
\begin{proof}
Let $\uu\in \undertilde{H}{}^1_0(\curl,\Omega)$ such that $(\nabla\nabla\times\uu,\nabla\nabla\times\uv)=(\uf,\uv)$ for any $\uv\in \undertilde{H}{}^1_0(\curl,\Omega)$, then $(0,\uu,\nabla\times\uu,0,\nabla\times\uu,0)$ solves the system \eqref{eq:mxdvfA'} with the same $\uf$. On the other hand, let $(m,\uu,\uphi,p,\ur,g)$ be the solution of \eqref{eq:mxdvfA'}. Then $\dv\uu=0$, $\nabla\times\uu=\uphi$, and it can be proved that $(\nabla\uphi,\nabla\upsi)=(\uf,\uv)$ for any $\upsi\in\undertilde{H}{}^1_0(\Omega)$ and $\upsi=\curl\uv$. The proof is completed.
\end{proof}

\begin{lemma}\label{lem:decomp}
The problem \eqref{eq:mxdvfA'} can be decomposed to the three subsystems and solved sequentially:
\begin{enumerate}
\item given $\uf$, solve for $\ur\in H_0(\curl,\Omega)$ and $g\in H^1_0(\Omega)$ that
\begin{equation}\label{eq:subA-1}
\left\{
\begin{array}{ccll}
(\nabla\times\ur,\nabla\times\uv) & -(\nabla g,\uv) &=-(\uf,\uv) &\uv\in H_0(\curl,\Omega)
\\
(\nabla n,\ur)& &=0&\forall\,n\in H^1_0(\Omega);
\end{array}
\right.
\end{equation}

\item with $\ur$ obtained, solve for $\uphi\in \undertilde{H}{}^1_0(\Omega)$ and $p\in L^2_0(\Omega)$ that 
\begin{equation}\label{eq:subA-2}
\left\{
\begin{array}{ccll}
(\nabla\uphi,\nabla\upsi) & +(p,\dv\upsi) &=-(\nabla\times\ur,\upsi)&\forall\,\upsi\in \undertilde{H}{}^1_0(\Omega)
\\
(\dv\uphi,q)&&=0&\forall\,q\in L^2_0(\Omega);
\end{array}
\right.
\end{equation}

\item with $\uphi$ obtained, solve for $\uu\in H_0(\curl,\Omega)$ and $m\in H^1_0(\Omega)$ that
\begin{equation}\label{eq:subA-3}
\left\{
\begin{array}{ccll}
(\nabla\times\uu,\nabla\times\us)&-(\nabla m,\us) &=(\uphi,\nabla\times\us)&\forall\,\us\in H_0(\curl,\Omega)
\\
(\nabla h,\uu)&&=0&\forall\,h\in H^1_0(\Omega).
\end{array}
\right.
\end{equation}
\end{enumerate}
\end{lemma}
\begin{proof}
We only have to show that the three subproblems are all well-posed, which can be verified by the stable Helmholtz decomposition of $H_0(\curl,\Omega)$. The proof is finished. 
\end{proof}

The theorem below constructs the regularity of the mixed system. 
\begin{theorem}\label{thm:regA}
Let $\Omega$ be a convex polyhedron, and $\uf\in (L^2(\Omega))^3$ such that $\dv\uf=0$. Let $(m,\uu,\uphi,p,\ur,g)$ be the solution of \eqref{eq:mxdvfA'}. Then
\begin{eqnarray}
&m=0, 
\\ 
&\uu\in \undertilde{H}{}^2(\Omega)\cap H_0(\curl,\Omega),\ \curl\uu\in \undertilde{H}{}^2(\Omega),
\\
&\uphi\in \undertilde{H}{}^2(\Omega)\cap \undertilde{H}{}^1_0(\Omega),\label{eq:regAphi}
\\
&p\in H^1(\Omega)\cap L^2_0(\Omega)\label{eq:regAp}
\\
&\ur\in \undertilde{H}{}^2(\Omega)\cap H_0(\curl,\Omega),\label{eq:regAr}
\\
&g=0.\label{eq:regAg}
\end{eqnarray}
\end{theorem}
We postpone the proof of Theorem \ref{thm:regA} after some technical lemmas.
\begin{lemma}(Sections 3.4 and 3.5 of \cite{GiraultRaviart1986})\label{lem:GRembd}
Let $\Omega$ be a convex polyhedron. Then
\begin{equation}
H_0(\dv,\Omega)\cap H(\curl,\Omega)\subset\undertilde{H}{}^1(\Omega),\ \ H_0(\curl,\Omega)\cap H(\dv,\Omega)\subset\undertilde{H}{}^1(\Omega).
\end{equation}
\end{lemma}
The lemma below is a smoothened analogue of Lemma \ref{lem:GRembd}, and we adopt the version in \cite{Chen.L;Wu.Y;Zhong.L;Zhou.J2016}. From this point onwards, $\lesssim$, $\gtrsim$, and $\cequiv$ respectively denote $\leqslant$, $\geqslant$, and $=$ up to a constant. The hidden constants depend on the domain, and, when triangulation is involved, they also depend on the shape-regularity of the triangulation, but they do not depend on $h$ or any other mesh parameter.
\begin{lemma}\label{lem:CWZZlem3.2}(Lemma 3.2 of \cite{Chen.L;Wu.Y;Zhong.L;Zhou.J2016})
For functions $\upsi\in H(\dv,\Omega)\cap H_0(\curl,\Omega)$ or $H_0(\dv,\Omega)\cap H(\curl,\Omega)$ satisfying $\curl\upsi\in \undertilde{H}{}^1(\Omega)$ and $\dv\upsi\in \undertilde{H}{}^1(\Omega)$. Then $\upsi\in \undertilde{H}{}^2(\Omega)$ and $\|\upsi\|_2\lesssim \|\curl\upsi\|_{1,\Omega}+\|\dv\upsi\|_{1,\Omega}$. 
\end{lemma}
\paragraph{\bf Proof of Theorem \ref{thm:regA}}
By Lemma \ref{lem:decomp}, we will show the regularity result by dealing with the systems \eqref{eq:subA-1}, \eqref{eq:subA-2} and \eqref{eq:subA-3} sequentially. 

Since $\dv\uf=0$, by \eqref{eq:subA-1}, it holds that $g=0$, and $(\nabla\times\ur,\nabla\times\uv)=(\uf,\uv)$ for any $\uv\in H_0(\curl,\Omega)$. Thus $\nabla\times\nabla\times\ur=\uf\in \undertilde{L}^2(\Omega)$, namely $\nabla\times\ur\in H(\curl,\Omega)$. Also, $\nabla\times\ur\in H_0(\dv,\Omega)$, and it follows that $\nabla\times\ur\in H(\curl,\Omega)\cap H_0(\dv,\Omega)\subset \undertilde{H}{}^1(\Omega)$. Note that by \eqref{eq:subA-1}, $\dv\ur=0$, thus by Lemma \ref{lem:CWZZlem3.2}, $\ur\in \undertilde{H}^2(\Omega)$. This proves \eqref{eq:regAr} and \eqref{eq:regAg}.

Substitute $\nabla\times\ur$ into the system \eqref{eq:subA-2}, standardly we obtain the estimate \eqref{eq:regAphi} and \eqref{eq:regAp}. We refer to \cite{Nicaise.S1997} for the regularity analysis of the 3D Stokes problem.

Further, substitute $\uphi$ into the system \eqref{eq:subA-3}, we obtain $m=0$, $\dv\uu=0$, $\uu\in \undertilde{H}{}^1(\Omega)\cap H_0(\curl,\Omega)$ and $\curl\uu\in \undertilde{H}{}^1(\Omega)$. By Lemma \ref{lem:CWZZlem3.2} again, it holds that $\uu\in\undertilde{H}{}^2(\Omega)$. Moreover, $\nabla\times\uu=\uphi$; this leads to $\nabla\times\uu\in\undertilde{H}{}^2(\Omega)$. 

Summing all above completes the proof. \qed
\begin{remark}
This theorem constructs the regularity result of the boundary value problem {\bf(A)} in forms \eqref{eq:bvpA} and \eqref{eq:vfA}, and confirms the regularity assumption of \cite{Sun.J2016} in the mid of Page 190 there. Moreover, by the same virtue, if further $\uf\in\undertilde{H}{}^1(\Omega)$, by Lemma \ref{lem:CWZZlem3.2}, we will obtain $\curl\ur\in\undertilde{H}{}^2(\Omega)$.
\end{remark}

\subsection{Mixedization of Problem (B$'$)}

We simply repeat the procedure for Problem {\bf (A$'$)}. Define 
$$
U:=H^1_0(\Omega)\times H_0(\curl,\Omega)\times \undertilde{H}{}^1_0(\Omega)\times H_0(\curl,\Omega)\times L^2_0(\Omega).
$$
The mixed variational problem is to find $(m,\uu,\uphi,\uu,p)\in U$, such that 
\begin{equation}\label{eq:mxdvpB'}
\left\{
\begin{array}{ccccccll}
&&&(\ur,\nabla n)&&& = 0, & n\in H^1_0(\Omega)
\\
&(\uu,\uv)&&-(\nabla\times\ur,\nabla\times\uv)&& &=(\uf,\uv),& \uv\in H_0(\curl,\Omega)
\\
&&(\nabla\uphi,\nabla\upsi)&+(\nabla\ur,\upsi)& &+(p,\dv\upsi) & =0, & \upsi\in \undertilde{H}{}^1_0(\Omega)
\\
(\nabla m,\us)&-(\nabla\times\uu,\nabla\times\us)& (\uphi,\nabla\times\us) & &&&=0, & \us\in H_0(\curl,\Omega)
\\
&&(\dv\uphi,q)&&&&=0, &q\in L^2_0(\Omega).
\end{array}
\right.
\end{equation}

\begin{lemma}
Given $\uf\in \undertilde{L}{}^2(\Omega)$, the problem \eqref{eq:mxdvpB'} admits a unique solution $(m,\uu,\uphi,\ur,p)\in U$. Moreover, 
\begin{equation}
\|m\|_{1,\Omega}+\|\uu\|_{\curl,\Omega}+\|\uphi\|_{1,\Omega}+\|\ur\|_{\curl,\Omega}+\|p\|_{0,\Omega}\leqslant C\|\uf\|_{(H_0(\curl,\Omega))'}.
\end{equation}
\end{lemma}
\begin{proof}
Again, we are going to verify Brezzi's conditions. Define $a((m,\uu,\uphi),(n,\uv,\upsi)):=(\uu,\uv)+(\nabla\uphi,\nabla\upsi)$, and $b((m,\uu,\uphi),(\us,q)):=(\nabla m,\us)-(\nabla\times\uu,\nabla\times\us)+(\uphi,\nabla\times\us)+(\dv\uphi,q)$. The continuity of $a(\cdot,\cdot)$ and $b(\cdot,\cdot)$ associated with $U$ follow immediately. 

Define $Z:=\{(m,\uu,\uphi)\in H^1_0(\Omega)\times H_0(\curl,\Omega)\times \undertilde{H}{}^1_0(\Omega):b((m,\uu,\uphi),(\us,q))=0,\forall\,(\us,q)\in H_0(\curl,\Omega)\times L^2_0(\Omega)\}$. Given $(m,\uu,\uphi)\in Z$, we have $m=0$ and $\uphi=\nabla\times\uu$. Thus the coercivity of $a(\cdot,\cdot)$ on $Z$ follows. 

Given $(\us,q)\in H_0(\curl,\Omega)\times L^2_0(\Omega)$, decompose $\us=\us{}_1+\us{}_2$ with $\us{}_1\in \nabla H^1_0(\Omega)$, and $\us{}_2\in(\nabla H^1_0(\Omega))^\perp$. Set $\uphi\in\undertilde{H}{}^1_0(\Omega)$, such that $(\dv\uphi,q)=(q,q)$ and $\|\uphi\|_1\leqslant C\|q\|_0$, $\nabla m=\us{}_1$, and $\uu\in (\nabla H^1_0(\Omega))^\perp$ such that $(\uphi-\nabla\times \uu,\nabla\times\uv)=(\nabla\times\us,\nabla\times\uv)$ for any $\uv\in H_0(\curl,\Omega)$. Then $b((m,\uu,\uphi),(\us,q))=(\us{}_1,\us{}_1)+(\nabla\times\us,\nabla\times\us)+(q,q)$, and $\|m\|_1+\|\uu\|_{\curl}+\|\uphi\|_{1,\Omega}\leqslant C(\|\us\|_{\curl,\Omega}+\|q\|_{0,\Omega})$. This leads to the inf-sup condition and completes the proof. 
\end{proof}
Similar to Lemma \ref{lem:equiv}, we can obtain the equivalence result below.
\begin{lemma}
The problem \eqref{eq:mxdvpB'} is equivalent to the primal problem {\bf (B$'$)}.
\end{lemma}

\begin{theorem}\label{thm:regB}
Let $\Omega$ be a convex polyhedron and $(m,\uu,\uphi,\ur,p)$ be the solution of \eqref{eq:mxdvpB'}. Then $m=0$, $\curl\uu\in \undertilde{H}{}^2(\Omega)$, $\uphi\in \undertilde{H}{}^2(\Omega)\cap \undertilde{H}{}^1_0(\Omega)$, $\ur\in \undertilde{H}{}^1(\Omega)\cap H_0(\curl,\Omega)$, $\curl\ur\in \undertilde{H}{}^1(\Omega)$ and $p\in L^2_0(\Omega)\cap H^1(\Omega)$. Further, if $\dv\uf=0$, then $\uu\in \undertilde{H}{}^2(\Omega)$.
\end{theorem}

\begin{proof}
By the stability of the system and since $\dv\ur=0$, we obtain $\ur\in \undertilde{H}{}^1(\Omega)$. As $(\nabla\times\ur,\nabla\times\uv)=-(\uf-\uu,\uv)$ for any $\uv\in H_0(\curl,\Omega)$, we have $\nabla\times\nabla\times\ur=\uf-\uu$, thus $\nabla\times\ur\in H(\curl,\Omega)\cap H_0(\dv,\Omega)\subset \undertilde{H}{}^1(\Omega)$. Then, standardly, we obtain $\uphi\in\undertilde{H}{}^2(\Omega)\cap \undertilde{H}{}^1_0(\Omega)$, and $p\in H^1(\Omega)\cap L^2_0(\Omega)$. By the second last line of the system, $m=0$ and $\nabla\times\uu=\uphi$; thus $\nabla\times\uu\in\undertilde{H}{}^2(\Omega)$. Further, if $\dv\uf=0$, then $\dv\uu=0$; this combined with that $\curl\uu\in\undertilde{H}{}^1_0(\Omega)$ leads to that $\uu\in\undertilde{H}{}^2(\Omega)$. The proof is completed.
\end{proof}
\begin{remark}
Theorem \ref{thm:regB} constructs regularity of {\bf (B)} in forms \eqref{eq:bvpB} and \eqref{eq:mxdvpB'}. It also confirms the validity of assumptions in Lemma 3.8 and Theorem 3.12 of \cite{Zheng.B;Hu.Q;Xu.J2011} by showing that $\uu\in\undertilde{H}{}^2(\Omega)$ and $\nabla\times\uu\in\undertilde{H}{}^2(\Omega)$ when $\uf\in \undertilde{L}{}^2(\Omega)$ with $\dv\uf=0$. 
The argument of this proof can be repeated onto that of Theorem \ref{thm:regA}. The difference between the two proofs is that we do not try to decompose \eqref{eq:mxdvpB'} to subsystems sequentially.  
\end{remark}

\section{Finite element discretizations of the mixed formulations}
\label{sec:dis}

For simplicity, we only consider the conforming finite element discretizations. Namely, we choose finite element spaces $H^1_{h0}\subset H^1_0(\Omega)$, $H_{h0}(\curl)\subset H_0(\curl,\Omega)$, $\undertilde{H}{}^1_{h0}\subset \undertilde{H}{}^1_0(\Omega)$ and $L^2_{h0}\subset L^2_0(\Omega)$, and use them to replace the respective Sobolev spaces when existing in the mixed variational form to generate a discretisation scheme. Particularly, the spaces $H^1_0(\Omega)$ and $H_0(\curl,\Omega)$ may appear more than once in the mixed formulation; this hints us to use different $H^{1,a}_{h0}$ and $H^{1,b}_{h0}$ and different $H^a_{h0}(\curl)$ and $H^b_{h0}(\curl)$ as their respective discretisation when convenient. In this section, we present some conditions of the well-posed-ness of the discretised system, construct generally its convergence analysis, and give some examples.

\subsection{Discretize Problem (A$'$)}

Let $H^{1,a}_{h0}\times H^a_{h0}(\curl)$ and $H^{1,b}_{h0}\times H^b_{h0}(\curl)$, identical or not, be two finite element subspaces of $H^1_0(\Omega)\times H_0(\curl,\Omega)$, and finite element space $\undertilde{H}{}^1_{h0}\times L^2_{h0}\subset \undertilde{H}{}^1_0(\Omega)\times L^2_0(\Omega)$. Define $V_h:=H^{1,a}_{h0}\times H^a_{h0}(\curl)\times \undertilde{H}{}^1_{h0}\times L^2_{h0}\times H^b_{h0}(\curl)\times H^{1,b}_{h0}$, and $V_h':=H^{1,b}_{h0}\times H^b_{h0}(\curl)\times \undertilde{H}{}^1_{h0}\times L^2_{h0}\times H^a_{h0}(\curl)\times H^{1,a}_{h0}$. The discretized formulation of {\bf (A$'$) } is to find $(m_h,\uu{}_h,\uphi{}_h,p_h,\ur{}_h,g_h)\in V_h$, such that, for any $(n_h,\uv{}_h,\upsi{}_h,q_h,\us{}_h,h_h)\in V_h'$,

\begin{equation}\label{eq:dismxdvfA'}
\left\{
\begin{array}{lll}
a(( m_h,\uu{}_h,\uphi{}_h),( n_h,\uv{}_h,\upsi{}_h))&+b(( p_h,\ur{}_h, g_h),(n_h,\uv{}_h,\upsi{}_h)) &= (\uf,\uv{}_h), 
\\
b(( m_h,\uu{}_h,\uphi{}_h), ( q_h,\us{}_h, h_h)) &&=0,
\end{array}
\right.
\end{equation}
where $a(\cdot,\cdot)$ and $b(\cdot,\cdot)$ follows the definitions \eqref{eq:aforA} and \eqref{eq:bforA}. The system is symmetric indefinite when $V_h=V_h'$, and unsymmetric otherwise. 

For the well-posedness of the discretised system, we set up some assumptions below.
\begin{description}
\item[A1] The exact relation holds: $\nabla H^{1,\alpha}_{h0}=\{\us{}_h\in H_{h0}^\alpha(\curl):\curl\us{}_h=0\},\  \alpha=a,b.$
\item[A2] The Poincar\'e inequality holds: $\|\us{}_h\|_{0,\Omega}\leqslant C\|\nabla\times\us{}_h\|_{0,\Omega}$ for $\us{}_h\in (\nabla  H^{1,\alpha}_{h0})^\perp$. Here $ (\nabla H^{1,\alpha}_{h0})^\perp$ is the orthogonal completion of $(\nabla H^{1,\alpha}_{h0})$ in $H_{h0}^\alpha(\curl)$ in $L^2$ inner product, $\alpha=a,b$.
\item[A3] The inf-sup condition holds: $\displaystyle\inf_{q_h\in L^2_{h0}\setminus\{0\}}\sup_{\upsi{}_h\in \undertilde{H}{}^1_{h0}\setminus\{\undertilde{0}\}}\frac{(\dv\upsi{}_h,q_h)}{\|q_h\|_{0,\Omega}\|\upsi{}_h\|_{1,\Omega}}\geqslant C$.
\end{description}

\begin{remark}
The assumptions {\bf A1} and {\bf A2} are imposed on the space pair $H^{1,\alpha}_{h0}$ and $H_{h0}^\alpha(\curl)$, $\alpha=a,b$, and the assumption {\bf A3} is imposed on the space pair $\undertilde{H}{}^1_{h0}$ and $L^2_{h0}$. This allows us to choose the three space pairs independently. Provided {\bf A3} itself, the system \eqref{eq:dismxdvfA'} can be decomposed to three subproblems and solved sequentially. 
\end{remark}

\begin{lemma}
Provided assumptions \textbf{A1}$\sim$\textbf{A3}, the problem \eqref{eq:dismxdvfA'} is well-posed on $V_h$.
\end{lemma}
\begin{proof}
When $H^{1,a}_{h0}=H^{1,b}_{h0}$ and $H^a_{h0}(\curl)=H^b_{h0}(\curl)$, namely $V_h=V_h'$, the proof is essentially the same as that of Lemma \ref{lem:wellposeA'}; otherwise, we can decompose the system to three subproblems and analyse the subproblems one by one, and the result can be proved.
\end{proof}

The convergence of the scheme is surveyed in the lemma below.
\begin{lemma}\label{lem:errA}
Let $(m,\uu,\uphi,p,\ur,g)\in V$ and $(m_h,\uu{}_h,\uphi{}_h,p_h,\ur{}_h,g_h)\in V_h$ be the solutions of \eqref{eq:mxdvfA'} and \eqref{eq:dismxdvfA'}, respectively. Provided assumptions {\bf A1}$\sim${\bf A3}, it holds that
\begin{enumerate}
\item $\displaystyle\|\ur-\ur{}_h\|_{\curl,\Omega}\leqslant C\inf_{\us{}_h\in H^b_{h0}(\curl)}\|\ur-\us{}_h\|_{\curl,\Omega};$
\item $\displaystyle \|\uphi-\uphi{}_h\|_{1,\Omega}+\|p-p_h\|_{0,\Omega}\leqslant C\big[\inf_{\upsi{}_h\in\undertilde{H}{}^1_{h0},q_h\in L^2_{h0}}(\|\uphi-\upsi{}_h\|_{1,\Omega}+\|p-q_h\|_{0,\Omega})+\|\ur-\ur{}_h\|_{0,\Omega}\big]$;
\item $\displaystyle \|\uu-\uu{}_h\|_{\curl,\Omega} \leqslant C\big[\inf_{\uv{}_h\in H^a_{h0}(\curl)}\|\uu-\uv{}_h\|_{\curl,\Omega}+\|\uphi-\uphi{}_h\|_{0,\Omega}\big]$;
\item $m_h=0=m$, $g_h=0=g$.
\end{enumerate}
\end{lemma}
\begin{proof}
Similar to Lemma \ref{lem:decomp}, we can decompose \eqref{eq:dismxdvfA'} to three subproblems and solve them sequentially. Then the lemma follows from the Ce\'a lemma and the second Strang lemma directly. We only have to note that, as $\dv\uf=0$, it follows that $g_h=0$ and $m_h=0$ by {\bf A1}, and since $g=0$ and $m=0$, $\inf_{h_h\in H^1_{h0}}\|g-h_h\|_{1,\Omega}=0$ and $\inf_{n_h\in H^1_{h0}}\|m-n_h\|=0$. 
\end{proof}
\begin{remark}
By Lemma \ref{lem:errA}, the error $\|m-m_h\|_{1,\Omega}+\|\uu-\uu{}_h\|_{\curl,\Omega}$ could be comparable to $\|\uphi-\uphi{}_h\|_{0,\Omega}$. This implies that when $H^{1,a}_{h0}$ and $H^a_{h0}(\curl)$ are chosen appropriately, a higher order convergence rate of $\uu$ may be expected. Indeed, for general $\uf$ (not assuming $\uf\in \undertilde{H}{}^1(\Omega)$), that we choose $H^a_{h0}(\curl)$ bigger than $H^b_{h0}(\curl)$ (meanwhile $H^{1,a}_{h0}$ bigger than $H^{1,b}_{h0}$) coincides with the higher smoothness of $\uu$ than that of $\ur$ on convex polyhedrons. However, if $H^{1,a}_{h0}\neq H^{1,b}_{h0}$, $V_h\neq V_h'$ and the problem is no longer symmetric, which may bring some extra difficulty. This is why we still discuss the choice $V_h=V_h'$ in, e.g., eigenvalue computation and other applications. 
\end{remark}

\subsection{Discretize Problem (B$'$)}

Define $U_h:=H^1_{h0}\times H_{h0}(\curl)\times\undertilde{H}{}^1_{h0}\times H_{h0}(\curl)\times L^2_{h0}$. The discretized mixed formulation is to find $(m_h,\uu{}_h,\uphi{}_h,\ur{}_h,p_h)\in U_h$, such that, for any $(n_h,\uv{}_h,\upsi{}_h,\us{}_h,q_h)\in U_h$,

\begin{equation}\label{eq:dismxdvpB'}
\left\{
\begin{array}{ccccccll}
&&&(\ur{}_h,\nabla n_h)&&& = 0 & 
\\
&(\uu{}_h,\uv{}_h)&&-(\nabla\times\ur{}_h,\nabla\times\uv{}_h)&& &=(\uf,\uv{}_h)&
\\
&&(\nabla\uphi{}_h,\nabla\upsi{}_h)&+(\nabla\times\ur{}_h,\upsi{}_h)& &+(p_h,\dv\upsi{}_h) & =0 & 
\\
(\nabla m_h,\us{}_h)&-(\nabla\times\uu{}_h,\nabla\times\us{}_h)& (\uphi{}_h,\nabla\times\us{}_h) & &&&=0 & 
\\
&&(\dv\uphi{}_h,q_h)&&&&=0. &
\end{array}
\right.
\end{equation}
The stability and the convergence of the scheme are surveyed in the two lemmas below.
\begin{lemma}
Provided assumptions \textbf{A1}$\sim$\textbf{A3}, the problem \eqref{eq:dismxdvpB'} is well-posed on $U_h$.
\end{lemma}

\begin{lemma}\label{lem:errB}
Let $(m,\uu,\uphi,\ur,p)$ and $(m_h,\uu{}_h,\uphi{}_h,\ur{}_h,p_h)$ be the solutions of \eqref{eq:mxdvpB'} and \eqref{eq:dismxdvpB'}, respectively. 
\begin{enumerate}
\item $m_h=0$;
\item $\displaystyle\|\uphi-\uphi{}_h\|_{1,\Omega}+\|p-p_h\|_{0,\Omega}
\\
\leqslant C\inf_{(n_h,\uv{}_h,\upsi{}_h,\us{}_h,q_h)\in U_h}[\|\uu-\uv{}_h\|_{\curl,\Omega}+\|\uphi-\upsi{}_h\|_{1,\Omega}+\|\ur-\us{}_h\|_{\curl,\Omega}+\|p-q_h\|_{0,\Omega}]$;
\item $\displaystyle\|\uu-\uu{}_h\|_{\curl,\Omega}+\|\ur-\ur{}_h\|_{\curl,\Omega}
\leqslant C\Big[
\inf_{\uv{}_h,\us{}_h\in H_{h0}(\curl)}(\|\uu-\uv{}_h\|_{\curl,\Omega}+\|\ur-\us{}_h\|_{\curl,\Omega})+\|\uphi-\uphi{}_h\|_{0,\Omega}\Big].
$
\end{enumerate}
\end{lemma}
\begin{proof}
The first item follows from {\bf A1}, and the second follows from the Cea lemma. Note that $(m,\uu,\ur)$ solves the problem
\begin{equation}\label{eq:reducedB}
\left\{
\begin{array}{cccll}
&&(\ur,\nabla n) & = 0 &\forall\,n\in H^1_{0}(\Omega)
\\
& (\uu,\uv)& -(\nabla\times \ur,\nabla\times\uv)& =(\uf,\uv)&\forall\,\uv\in H_0(\curl,\Omega)
\\
(\nabla m,\us)&-(\nabla\times\uu,\nabla\times\us)&&=-(\uphi,\nabla\times\us)&\forall\,\us\in H_0(\curl,\Omega),
\end{array}
\right.
\end{equation}
and $(m_h,\uu{}_h,\ur{}_h)$ solves the finite element problem:
\begin{equation}\label{eq:subproblemBdis}
\left\{
\begin{array}{cccll}
&&(\ur{}_h,\nabla n_h) & = 0 &\forall\,n_h\in H^1_{h0}
\\
& (\uu{}_h,\uv{}_h)& -(\nabla\times\ur{}_h,\nabla\times\uv{}_h)& =(\uf,\uv{}_h)&\forall\,\uv{}_h\in H_{h0}(\curl)
\\
(\nabla m_h,\us{}_h)&-(\nabla\times\uu{}_h,\nabla\times\us{}_h)&&=-(\uphi{}_h,\nabla\times\us{}_h)&\forall\,\us{}_h\in H_{h0}(\curl).
\end{array}
\right.
\end{equation}
The third item follows from the Ce\'a lemma and the second Strang lemma.
\end{proof}

Lemma \ref{lem:errB} reveals that the error $\|\uu-\tilde{\uu}{}_h\|_{\curl,\Omega}+\|\ur-\tilde{\ur}{}_h\|_{\curl,\Omega}$ can be comparable with $\|\uphi-\uphi{}_h\|_0$. This hints us to use some bigger $H_{h0}(\curl)$ to expect higher accuracy of $\uu$ and $\ur$ than that of $\uphi$ and $p$ for smooth $\uu$ and $\ur$.

\subsection{Examples of finite element quartos}

\subsubsection{Problem {\bf (A$'$)}}
For problem {\bf (A$'$)}, we choose 
\begin{equation}
H^{1,a}_{h0}:=\mathcal{L}^2_{h0},\, H^a_{h0}(\curl):=N^2_{h0},\, \undertilde{H}{}^1_{h0}:=(\mathcal{L}^2_{h0})^2,\, H^{1,b}_{h0}:=\mathcal{L}^1_{h0},\, H^b_{h0}(\curl):=N^1_{h0},\, L^2_{h0}:=\mathcal{L}^1_{h}\cap L^2_0(\Omega).
\end{equation}
The assumptions {\bf A1}$\sim${\bf A3} can be verified. Particularly, {\bf A1} and {\bf A3} can be found in \cite{Boffi.D;Brezzi.F;Fortin.M2013}, and {\bf A2} can be found in \cite{Hiptmair.R2002,AFW2006,Kikuchi.F1989}. Its convergence then follows from Lemma \ref{lem:errA}. We here present a specific estimate on convex polyhedrons.
\begin{lemma}\label{lem:errestA}
Let $\Omega$ be a convex polyhedron and $\uf\in\undertilde{L}^2$, $\dv\uf=0$. Let $(m,\uu,\uphi,p,\ur,g)\in V$ and $(m_h,\uu{}_h,\uphi{}_h,p_h,\ur{}_h,g_h)\in V_h$ be the solutions of \eqref{eq:mxdvfA'} and \eqref{eq:dismxdvfA'}, respectively. Then 
\begin{enumerate}
\item $g_h=0=g$ and $m_h=0=m$;
\item $\displaystyle\|\ur-\ur{}_h\|_{\curl,\Omega}\leqslant Ch(\|\ur\|_{1,\Omega}+\|\curl\ur\|_{1,\Omega})\leqslant Ch\|\uf\|_{0,\Omega}$;
\item $\displaystyle \|\uphi-\uphi{}_h\|_{1,\Omega}+\|p-p_h\|_{0,\Omega}\leqslant Ch(\|\uphi\|_{2,\Omega}+\|p\|_{1,\Omega}+\|\ur\|_{1,\Omega}+\|\curl\ur\|_{1,\Omega})\leqslant Ch\|\uf\|_{0,\Omega}$;
\item $\|\uphi-\uphi{}_h\|_0\leqslant Ch^2\|\uf\|_{0,\Omega}$;
\item $\displaystyle \|\uu-\uu{}_h\|_{\curl,\Omega}\leqslant C(h^2(\|\uu\|_{2,\Omega}+\|\curl\uu\|_{2,\Omega})+\|\uphi-\uphi{}_h\|_{0,\Omega}) \leqslant Ch^2\|\uf\|_{0,\Omega}$.
\end{enumerate}
\end{lemma}
\begin{proof}
We only prove the estimate $\|\uphi-\uphi{}_h\|_{0,\Omega}$ by dual argument, and the remaining follows from Lemma \ref{lem:errA} directly.

Define $\hat{V}_h:=H^{1,b}_{h0}\times H^b_{h0}(\curl)\times \undertilde{H}{}^1_{h0}\times L^2_{h0}\times H^b_{h0}(\curl)\times H^{1,b}_{h0}$, and let $(\hat{m}_h,\hat{\uu}{}_h,\hat{\uphi}{}_h,\hat{p}_h,\hat{\ur}{}_h,\hat{g}_h)\in \hat{V}_h$ be so that  
\begin{equation}\label{eq:dualsystem}
\left\{
\begin{array}{lll}
a((\hat m_h,\hat\uu{}_h,\hat\uphi{}_h),(\hat n_h,\hat\uv{}_h,\hat\upsi{}_h))&+b((\hat p_h,\hat\ur{}_h,\hat g_h),(\hat n_h,\hat\uv{}_h,\hat\upsi{}_h)) &= (\uf,\hat\uv{}_h), 
\\
b((\hat m_h,\hat\uu{}_h,\hat\uphi{}_h), (\hat q_h,\hat\us{}_h,\hat h_h)) &&=0.
\end{array}
\right.
\end{equation}
for any $(\hat n_h,\hat\uv{}_h,\hat\upsi{}_h,\hat q_h,\hat\us{}_h,\hat h_h)\in\hat{V}_h$, where $a(\cdot,\cdot)$ and $b(\cdot,\cdot)$ follows the definitions \eqref{eq:aforA} and \eqref{eq:bforA}.Then $\hat{m}_h=\hat{g}_h=0$, and $(\hat{\uphi}{}_h,\hat{p}_h,\hat{\ur}{}_h)=({\uphi}{}_h,{p}_h,{\ur}{}_h)$. Let $(\tilde m,\tilde\uu,\tilde\uphi,\tilde p,\tilde\ur,\tilde g)\in V$ solve the variational problem
\begin{equation}\label{eq:dualsystem}
\left\{
\begin{array}{lll}
a((\tilde m,\tilde\uu,\tilde\uphi),(n,\uv,\upsi))&+b((\tilde p,\tilde\ur,\tilde g),(n,\uv,\upsi)) &= (\uphi-\uphi{}_h,\upsi), 
\\
&& \quad (n,\uv,\upsi)\in H^1_0(\Omega)\times H_0(\curl,\Omega)\times \undertilde{H}{}^1_0(\Omega)
\\
b((\tilde m,\tilde\uu,\tilde\uphi), (q,\us,h)) &&=0
\\
&&\quad (q,\us,h)\in L^2_0(\Omega)\times H_0(\curl,\Omega)\times H^1_0(\Omega).
\end{array}
\right.
\end{equation}
Then $\tilde\ur=\undertilde{0}$, $\tilde g=0$, $\tilde m=0$, and, by the same virtue as that of Theorem \ref{thm:regA},
$$
\|\tilde\uphi\|_{2,\Omega}+\|\tilde p\|_{1,\Omega}+\|\tilde \uu\|_{2,\Omega}+\|\curl\tilde\uu\|_{2,\Omega}\leqslant C\|\uphi-\uphi{}_h\|_{0,\Omega}.
$$
Thus substituting $(n,\uv,\upsi,q,\us,h)=(m-\hat m_h,\uu-\hat\uu{}_h,\uphi-\uphi{}_h,p-p_h,\ur-\ur{}_h,g-\hat g_h)$ into \eqref{eq:dualsystem}, we have 
\begin{multline*}
\|\uphi-\uphi{}_h\|_{0,\Omega}^2=a((\tilde m,\tilde\uu,\tilde\uphi),(m-\hat m_h,\uu-\hat\uu{}_h,\uphi-\uphi{}_h))
\\
+b((\tilde p,\tilde\ur,\tilde g),(p-\hat p_h,\ur-\ur{}_h,g-\hat g_h))+b((\tilde m,\tilde\uu,\tilde\uphi), (p-\hat p_h,\ur-\ur{}_h,g-\hat g_h)).
\end{multline*}
Further, for any $(\hat n_h,\hat\uv{}_h,\hat\upsi{}_h,\hat q_h,\hat\us{}_h,\hat h_h)\in\hat{V}_h$, the orthogonality holds that
\begin{multline*}
\|\uphi-\uphi{}_h\|_{0,\Omega}^2
=a((\tilde m-\hat n_h,\tilde\uu-\hat\uv{}_h,\tilde\uphi-\hat\upsi{}_h),(m-\hat m_h,\uu-\hat\uu{}_h,\uphi-\uphi{}_h))+
\\
b((\tilde p-\hat q_h,\tilde\ur-\hat\us{}_h,\tilde g-\hat h_h),(p-\hat p_h,\ur-\ur{}_h,g-\hat g_h))+b((\tilde m-\hat n_h,\tilde\uu-\hat\us{}_h,\tilde\uphi), (p-\hat p_h,\ur-\ur{}_h,g-\hat g_h)).
\end{multline*}
Thus
\begin{multline*}
\|\uphi-\uphi{}_h\|_{0,\Omega}^2\leqslant C(\|\tilde m-\hat m_h\|_{0,\Omega}+\|\tilde\uu-\hat\uu{}_h\|_{\curl,\Omega}+\|\tilde\uphi-\hat\uphi{}_h\|_{1,\Omega}+\|\tilde p-\hat p_h\|_{0,\Omega}+\|\tilde\ur-\hat\ur{}_h\|_{\curl,\Omega})
\\
\times \inf_{(\hat n_h,\hat\uv{}_h,\hat\upsi{}_h,\hat q_h,\hat\us{}_h,\hat h_h)\in\hat{V}_h}
(\|\tilde m-\hat n_h\|_{0,\Omega}+\|\tilde\uu-\hat\uv{}_h\|_{\curl,\Omega}+\|\tilde\uphi-\hat\upsi{}_h\|_{1,\Omega}+\|\tilde p-\hat q_h\|_{0,\Omega}+\|\tilde\ur-\hat\us{}_h\|_{\curl,\Omega}).
\end{multline*}
Then by finite element estimate,
\begin{equation}
\|\uphi-\uphi{}_h\|_{0,\Omega}^2\leqslant Ch\|\uf\|_{0,\Omega}\cdot h\|\uphi-\uphi{}_h\|_{0,\Omega},
\end{equation}
which leads to that $\|\uphi-\uphi{}_h\|_{0,\Omega}\leqslant Ch^2\|\uf\|_{0,\Omega}$. This completes the proof.
\end{proof}

\begin{remark}
The convergence analysis with respect to the regularities of $\uu$, $\ur$, and etc. follows directly from Lemma \ref{lem:errA}. We here construct a convergence analysis with respect to $\|\uf\|_{0,\Omega}$, which can exploit the regularity of solution functions to a full extent with economical complexity. When further $\uf\in\undertilde{H}{}^1$, we can set $H^{1,b}_{h0}:=\mathcal{L}^2_{h0}$ and $H^b_{h0}(\curl):=N^2_{h0}$, and obtain higher accuracy of $\|\ur-\ur{}_h\|_{\curl,\Omega}$.
\end{remark}

\subsubsection{Problem {\bf (B$'$)}}
For problem {\bf (B$'$)}, we choose 
\begin{equation}
H^1_{h0}:=\mathcal{L}^2_{h0},\, H_{h0}(\curl):=N^2_{h0},\, \undertilde{H}{}^1_{h0}:=(\mathcal{L}^2_{h0})^2,\,  L^2_{h0}:=\mathcal{L}^1_{h}\cap L^2_0(\Omega).
\end{equation}
The assumptions {\bf A1}$\sim${\bf A3} can be verified. For the convergence rate, we have the lemma below.

\begin{lemma}
Let $\Omega$ be a convex polyhedron and $\uf\in\undertilde{L}^2$. Let $(m,\uu,\uphi,\ur,p)\in U$ and $(m_h,\uu{}_h,\uphi{}_h,\ur{}_h,p_h)\in U_h$ be the solutions of \eqref{eq:mxdvpB'} and \eqref{eq:dismxdvpB'}, respectively. Then 
\begin{enumerate}
\item $m_h=0=m$;
\item $\displaystyle \|\uphi-\uphi{}_h\|_{1,\Omega}+\|p-p_h\|_{0,\Omega}\leqslant Ch\|\uf\|_{0,\Omega}$, $\|\uphi-\uphi{}_h\|_0\leqslant Ch^2\|\uf\|_{0,\Omega}$;
\item $\displaystyle\|\ur-\ur{}_h\|_{\curl,\Omega}+\|\uu-\uu{}_h\|_{\curl,\Omega}\leqslant  Ch\|\uf\|_{0,\Omega}$;
\item if $\uf\in \undertilde{H}{}^1(\Omega)$, then $\displaystyle \|\uu-\uu{}_h\|_{\curl,\Omega}+\|\ur-\ur{}_h\|_{\curl,\Omega}\leqslant Ch^2\|\uf\|_{1,\Omega}$.
\end{enumerate}
\end{lemma}
\begin{proof}
Similar to Lemma \ref{lem:errestA}, we only have to construct the estimate $\|\uphi-\uphi{}_h\|_{0,\Omega}$, which can be done by repeating the same dual argument as that in the proof of Lemma \ref{lem:errestA}, and we omit it here. The remaining follows from Lemma \ref{lem:errB} directly. The proof is completed.  
\end{proof}

\section{Concluding remarks}
\label{sec:con}

In this paper, we study the fourth order curl problem, and develop for them amiable mixed schemes. We construct equivalent mixed formulation of the two variant model problems; the mixed formulations are stable on standard $H(\curl)$ and $H^1$ spaces. Regularity results are constructed for the mixed formulations, and then the primal ones. Existing finite element quartos that satisfy quite mild assumptions can then lead to stable and convergent discretisation scheme of the model problem. Some finite element examples which are optimal with respect to the regularity are given. The newly developed schemes are friendly in building, analysing and designing optimal solvers. The scheme can be implemented with various familiar finite element packages. 

In the derivation of the equivalent formulations, we figure out the boundary condition that $\nabla\times\uu$ has to satisfy. Particularly, $(\nabla\times\uu)\cdot\mathbf{n}=0$ is a condition which should be satisfied by the variable $\nabla\times\uu$ essentially; also the condition can not be imposed weakly with the duality operated on $H(\curl,\Omega)$. This forces us to bring $\undertilde{H}{}^1_0(\Omega)$ which has bigger capacity for boundary conditions into discussion, and Lemma \ref{lem:h=cd} guarantees the equivalence. 

As only spaces of low smoothness and of low degrees are involved, it is possible to construct finite element discretizations that are nested algebraically and topologically on nested grids; this will provide convenience in designing multilevel methods, and can be utilised in practice (c.f., e.g.,\cite{Chen.L;Hu.X;Wang.M;Xu.J2015,Chen.L2015,Zhang.S;Xi.Y;Ji.X2016,Li.Z;Zhang.S2016}). In this paper, we only consider two familiar variants boundary value problems. In many contexts, second order operators also appear in the boundary value problem, and parameters of various scale may appear in front of operators of different orders; see the model problem in \cite{Zheng.B;Hu.Q;Xu.J2011}. Designing parameter-robust discretisation is interesting and practically important, and will be discussed in future. Also, we focus ourselves on source problems in the present paper. The utilisation of the mixed scheme presented in this paper onto eigenvalue computation and analysis, especially in designing multilevel algorithm (c.f.  \cite{Zhang.S;Xi.Y;Ji.X2016}), will be discussed in future.

\end{document}